\documentclass[11pt,a4paper]{amsart}
\usepackage{amssymb,amsmath}
\usepackage{graphicx}
\usepackage{subcaption}
\usepackage{comment}
\usepackage{tikz}
\usetikzlibrary{arrows}

\def\({\left(}
\def\){\right)}
\def\Nx{\nabla_x}
\def\Cal{\mathcal}

\def\eb{\varepsilon}

\def\tilde{\widetilde}
\newcommand{\be}{\begin{equation} }
\newcommand{\ee}{\end{equation} }

\def \and{\qquad\text{and}\qquad}

\def\Bbb{\mathbb}
\def\Dt{\partial_t}
\def\Dx{\Delta_x}

\def\({\left(}
\def\){\right)}
\def\Nx{\nabla_x}

\def\eb{\varepsilon}
\def\Cal{\mathcal}
\def\eb{\varepsilon}

\def\R {\mathbb R}

\def\<{\left<}
\def\>{\right>}

\def \and{\qquad\text{and}\qquad}

\def\Bbb{\mathbb}
\def\Dt{\partial_t}
\def\Dx{\Delta_x}
\def\R {\mathbb R}

\newtheorem{proposition}{Proposition}[section]
\newtheorem{theorem}[proposition]{Theorem}
\newtheorem{corollary}[proposition]{Corollary}

\theoremstyle{definition}
\newtheorem{definition}[proposition]{Definition}
\newtheorem{remark}[proposition]{Remark}
\newtheorem{example}[proposition]{Example}

\numberwithin{equation}{section}
\def\be{\begin{equation}}
\def\ee{\end{equation}}
\def\bp{\begin{proof}}
\def\ep{\end{proof}}

\def \no#1#2#3 {{\bf #1} (#3), #2.}
\def \eds#1#2#3 {#1, #2, #3.}

\title[Attractors for multi-dimensional time]
{Attractors for semigroups with multi-dimensional time and PDEs in unbounded domains }
\author[ A. Kostianko,  and  S. Zelik]
{Anna Kostianko${}^{1,2}$, and Sergey Zelik${}^{1,3,4}$}

\address{${}^1$  Lanzhou University, Lanzhou\\ 730000,
P.R. China.}
\address{${}^2$  Imperial College, London SW7 2AZ, United Kingdom.}

\address{${}^3$ University of Surrey, Guildford, GU2 7XH, United Kingdom.}

\address{${}^4$ Keldysh Institute of Applied Mathematics, Moscow, Russia}

\email{a.kostianko@imperial.ac.uk}
\email{s.zelik@surrey.ac.uk}

\begin{document}

\begin{abstract} We develop the attractors theory for the semigroups with multidimensional time belonging to some closed cone in an Euclidean space and apply the obtained general results to partial differential equations (PDEs) in unbounded domains. The main attention is payed to elliptic boundary problems in general unbounded domains. In contrast to the previous works in this direction our theory does not require the underlying domain to be cylindrical  or cone-like or to be shift semi-invariant with respect to some direction. In particular, the theory is applicable to the exterior domains. 
\end{abstract}

\thanks{This work is partially supported by the grant 19-71-30004 of RSF (Russia),
  and the Leverhulme grant No. RPG-2021-072 (United Kingdom)}
\maketitle
\tableofcontents
\section{Introduction}
It is well-known, see \cite{KH95,Shi01} and references therein, that even relatively simple systems of ODEs are usually cannot be solved in elementary functions and may generate  non-trivial dynamics, so the qualitative theory is required in order to understand the longtime behavior of their solutions. The central role in this study is played by the concept of a {\it dynamical system} (DS). To associate a DS with the considered ODEs, we need to find the space $\Phi$ where the initial data lives and where the initial value problem is globally well-posed forward in time $t$ (which is referred as a phase space of the considered problem) and introduce the solution operator $S(t):\Phi\to\Phi$, which is defined by $S(t)u_0=u(t)$ where $u_0\in\Phi$ is the initial state (at $t=0$) of our system and $u(t)$ is the solution at time moment $t$. These operators, which describe the evolution of the considered system, are often referred as a DS associated with the initial system of ODEs. In the case, where this system is autonomous, these operators generate a semigroup in the phase space $\Phi$:
$$
S(0)=\operatorname{Id},\ \  S(t_1+t_2)=S(t_1)\circ S(t_2), \ \ t_1,t_2\in\R_+.
$$
The analogous approach works also for the case where the DS is generated by an evolutionary partial differential equation (PDE) with a natural difference that the phase space $\Phi$ becomes an infinite-dimensional space (e. g, the space of square integrable functions). An important particular case where the dynamical systems approach is especially effective is the case of {\it dissipative} PDEs, which includes many important classes of physically relevant equations such as reaction-diffusion and Navier-Stokes systems, pattern formation equations, damped wave equations, etc. Roughly speaking, such systems lose and gain energy by interacting with the external world, the energy income  usually takes place  though lower Fourier modes and the energy dissipation becomes dominating in higher Fourier modes. Then the energy flow from lower to higher modes caused by the nonlinearity is responsible for the formation of complex space-time structures, the so-called {\it dissipative structures} in the terminology of Prigogine, see \cite{Pri77}. A mathematically rigorous definition of a dissipative system is usually related with the so-called dissipative estimate which, in the case where $\Phi$ is a subspace of a normed space, can be written down as follows:
\begin{equation}\label{0.dis}
\|S(t)u_0\|_{\Phi}\le Q(\|u_0\|_{\Phi})e^{-\alpha t}+C_*,\ \ t\ge0,\ \ u_0\in\Phi,
\end{equation}
where the positive constant $\alpha$ and a monotone increasing function $Q$ are independent of $u_0$ and $t$. Note that some kind of asymptotic compactness is also often included  into the definition, see \cite{BV92,Hal88,tem} and references therein. The norm $\|u_0\|_\Phi$ plays the role of energy here and the dissipative estimate \eqref{0.dis} claims that the dissipation becomes dominating on higher energy levels and energy income is essential only for relatively low energy levels.
\par
The concept of an {\it attractor} plays a central role in the modern theory of dissipative systems. By the definition, an attractor $\Cal A$ is a compact strictly invariant set in the phase space $\Phi$ which attracts the images of all "bounded" sets in the phase space as time tends to infinity. In relatively simple cases,  "bounded" sets are determined using the norm in the phase space $\Phi$, but in more general situation, they are defined by fixing a {\it bornology} $\Bbb B$ on $\Phi$, which consists of subsets of $\Phi$, which are referred as "bounded" sets, see \cite{Z23} and references therein. Thus, on the one hand, the attractor consists of all trajectories $u(t)$, $t\in\R$, of the considered DS which are defined and bounded and, therefore, somehow captures all of the non-trivial longtime dynamics. On the other hand, it is usually much smaller than the initial phase spaces (recall that compact sets in infinite-dimensional normed spaces are nowhere dense and are almost finite-dimensional). The attractors theory is most effective in the case of dissipative PDEs in bounded domains, where in many cases the fractal dimension of the attractor is finite, so keeping in mind the Man\'e projection theorem, this theory allows us to build up the finite-dimensional reduction of the initial PDE to the reduced system of ODEs on the attractor, the so-called inertial form of the initial PDE, see \cite{BV92,ChVi02,Rob,Rob11,tem} and references therein.
\par
The situation becomes much more complicated when a PDE in an unbounded spatial domain is considered. In this case, together with the unbounded temporal direction, we have also unbounded spatial directions, so together with complicated dynamics in time (temporal chaos), we also have complicated spatial behavior of solutions (the so-called spatial chaos). Also, the interaction between spatially and temporary chaotic modes may generate the so-called spatio-temporal chaos, see \cite{MZ08,Z04} and references therein. Although the attractors theory for dissipative PDEs in unbounded domains is highly developed nowadays, see \cite{MZ08} and references therein, the one-dimensional time does not look as an appropriate tool to handle e.g. spatial or spatio-temporal chaos, so it was suggested in \cite{Z04} to interpret both $t$ and $x$ directions as a multi-dimensional "time" and consider a multi-parametrical semigroup $S(t,h):=S(t)\circ \Cal T(h)$, $t\in\R_+$, $h\in\R^d$, which acts on the phase space $\Phi$ and which consists of the standard temporal evolution $S(t)$ and the group of spatial shifts $\Cal T(h)$. This semigroup is then interpreted as a DS  with multi-dimensional time and its dynamical quantities (such as topological entropy, directional entropies, etc.) are studied in relation with spatial and spatio-temporal chaos, see \cite{Z04}.
\par
However, to the best of our knowledge, the attractors theory for DS with multi-dimensional time was not considered before and the ultimate goal of this paper is to develop such a theory. In section \ref{s1}, we start with a semigroup $S(h):\Phi\to\Phi$ acting on a metric space $\Phi$, where the "time" $h$ belongs to some closed cone $\Cal C\subset\R^m$ with non-empty interior. In addition, we introduce an admissible set of times $\Sigma\subset\Cal C$ (referred also as a time arrow) such that $\partial\Cal C\subset \partial\Sigma$ together with some bornology $\Bbb B$ for which no extra assumptions are posed and extend the main concepts of the attractor theory to this case. For instance, the set $K$ is $(\Bbb B,\Sigma)$-attracting set if for every $B\in\Bbb B$ and every $\Cal O(K)$ there exists $D=D(B,\Cal O)$ such that
$$
S(h)B\subset\Cal O(K),\ \ d(h,\partial\Sigma)\ge D.
$$ 
 Thus, "big" time $h$ is determined by the condition $d(h,\partial\Sigma)\gg1$. The most natural choice of $\Sigma$ is $\Sigma=\Cal C$, but keeping in mind our main application to elliptic boundary value problems in unbounded non-cylindrical and non-conical domains, it is natural to consider not all times  $h\in\Cal C$ in the attraction property, but only times $h\in\Sigma$. Since we do not have further restrictions on $\Sigma$, the structure of multi-dimensional time arrow may be curvilinear, have holes, etc. The condition that $\partial\Cal C\subset\partial\Sigma$ is included in order to have the description of the attractor in terms of bounded complete trajectories defined for all $h\in\R^m$.
\par
With the changes described above, the attractors theory for this multi-dimensional time case looks very similar to the case of standard time. For instance, the $(\Bbb B,\Sigma)$-attractor is defined as a compact strictly invariant $(\Bbb B,\Sigma)$-attracting set of the phase space and the main theorem tells us that such an attractor exists if the semigroup $S(h)$ possesses a compact $(\Bbb B,\Sigma)$-attracting set and $S(h):\Phi\to\Phi$ are continuous for every fixed $h\in\Cal C$. We also give (in section \ref{s1}) the sufficient conditions for the attractor to be represented through the set of all complete bounded trajectories $u(s)$ of $S(h)$ which are defined for all $s\in\R^m$.
\par
In section \ref{s2}, we apply the abstract results obtained in section \ref{s1} to PDEs in unbounded domains. Although we consider a parabolic PDE as one of our examples, our main source of motivation is related with elliptic PDEs in general unbounded domains. The applications of the attractors theory to such problems are usually based on the notion of the so-called {\it trajectory attractor} and have been initiated in \cite{Ba94,VZ96} and \cite{Z-VIN,Z98} for the cases of cylindrical domains and cone like domains respectively. In both cases, the considered domain $\Omega$ was semi-invariant with respect to shifts in some spatial direction $\vec l$ and exactly this direction is interpreted as the direction of a "temporal" evolution. Then, if we introduce the space $\Cal K_+$ of all solutions of the considered problem for all possible boundary data, then the semigroup of "temporal" translations $T(t):\Cal K_+\to\Cal K_+$, $(T(t)u)(x):=u(x+t\vec l)$, $t\ge0$, will act on this space. After that the DS $(T(t),\Cal K_+)$ is interpreted as a trajectory dynamical system associated with the considered elliptic boundary value problem and the corresponding attractors (which are called trajectory attractors) are studied, see also \cite{ChVi95,ChVi02} for a general theory of trajectory attractors.
\par
The principal drawback of this scheme is that it requires the existence of a direction with respect to which the domain $\Omega$ is semi-invariant. This assumption looked unavoidable, and to the best of our knowledge, the attractors theory has been never applied before to elliptic problems in the domain without such a direction, for instance, for equations in the exterior domains. Surprisingly, the theory of attractors developed in this paper works perfectly for elliptic boundary problems in general unbounded domains without any semi-invariant directions, in particular, for exterior domains. This is achieved mainly due to introducing the time arrow $\Sigma$ which does not require any invariance properties, see Examples \ref{Ex2.ell-cone}, \ref{Ex2.ext} and \ref{Ex2.disc} below for more details.

\section{Main results}\label{s1}
In this section, we build up the attractor theory for multi-parametric semigroups $S(h):\Phi\to\Phi$, where $\Phi$ is a  metric space and the "time" variable $h$ is multi-dimensional. Namely, we assume that we are given a closed cone $\Cal C\subset\R^m$ with non-empty interior. We also assume that the operators $S(h)$ are well-defined  for any $h\in\Cal C$ and satisfy the semigroup identity:
\begin{equation}\label{1.sem}
S(0)=\operatorname{Id},\ \ S(h_1+h_2)=S(h_1)\circ S(h_2),\ \ h_1,h_2\in\Cal C.
\end{equation}
We now need to fix the proper extensions of the standard concepts of the attractors theory to the considered multi-dimensional case. First of all, we fix a bornology $\Bbb B$ on $\Phi$, i.e. the collection of subsets of $\Phi$, which will be attracted by the desired attractor and which will be further referred as "bounded" sets of $\Phi$, and some non-empty set $\Sigma\subset\Cal C$ such that $\partial\Cal C\subset\partial\Sigma$, which will be referred as an admissible set or time arrow,  and then we define what are absorbing and attracting sets for our situation.

\begin{definition}\label{Def1.aa} Let the semigroup $S(h)$ satisfy the above assumptions. Then the set $\Cal B\subset \Phi$ is an absorbing set ($(\Bbb B,\Sigma)$-absorbing set) if, for every bounded set $B\in\Bbb B$, there exists $D=D(B)$ such that
\begin{equation}\label{1.abs}
S(h)B\subset\Cal B,\ \ \text{for all}\ \ h\in\Cal C\ \ \text{satisfying }\ d(h,\partial\Sigma)\ge D,
\end{equation}
where $d(h,V)$ stands for the standard distance in $\R^n$.
\par
Analogously, the set $K$ is an attracting set for this semigroup if for every bounded set $B\in\Bbb B$ and every neighbourhood $\Cal O(K)$, there exists $D=D(B,\Cal O)$ such that
\begin{equation}\label{1.atr}
S(h)B\subset\Cal O(K),\ \ \text{for all}\ \ h\in\Cal C\ \ \text{satisfying }\ d(h,\partial\Sigma)\ge D,
\end{equation}
In particular, the bornology $\Bbb B$  on $\Phi$ can be determined by the fixed metric in $\Phi$ (which leads to the standard global attractors), or by the metric of a different metric space $\Psi$ (which leads to the so-called $(\Psi,\Phi)$-attractors in the terminology of Babin and Vishik, see \cite{BV92}), but in general the bornology $\Bbb B$ may be an arbitrary collection of subsets of $\Phi$. The time arrow $\Sigma$  typically coincides with the initial cone   $\Cal C$, but in a more general situation may be an arbitrary subset of $\Cal C$ satisfying $\partial \Cal C\subset\partial\Sigma$ and the following condition: for every $D>0$, there exists a point $h_D\in\Sigma$ such that
\begin{equation}\label{2.sigma}
d(h_D,\partial\Sigma)\ge D.
\end{equation}
\end{definition}
The extension of the concept of a global attractor now reads.
\begin{definition}\label{Def1.attr} Let $\Phi$ be a metric space, $\Bbb B$ be a bornology on $\Phi$ and the semigroup $S(h)$ satisfy the above assumptions. Then, a set $\Cal A\subset\Phi$ is a $(\Bbb B,\Sigma)$-attractor for the semigroup $S(h)$ if the following conditions are satisfied:
\par
1. The set $\Cal A$ is  compact set in $\Phi$;
\par
2. The set $\Cal A$ is strictly invariant, i.e. $S(h)\Cal A=\Cal A$ for all $h\in\Cal C$;
\par
3. The set $\Cal A$ is $(\Bbb B,\Sigma)$-attracting for the semigroup $S(t)$.
\end{definition}
In order to extend the representation formula for the attractor (i.e. the description of the attractor in terms of all bounded complete solutions), we need one more definition.
\begin{definition}Let the semigroup $S(t)$ satisfy the above assumptions. Then a function $u:\R^m\to\Phi$ is a complete bounded trajectory of the semigroup $S(h)$ if
\par
1. $S(h)u(s)=u(h+s)$ for all $h\in\Cal C$ and $s\in\R^m$;
\par
2. There exists a bounded set $B_u\in\Bbb B$ such that $u(s)\in B_u$ for all $s\in\R^m$.
\par
The set of all complete bounded trajectories of $S(h)$ is denoted by $\Cal K$.
\end{definition} 
\begin{remark}\label{Rem1.str} Although the extension to the case of multi-dimensional time constructed above looks straightforward, the unusual form of the attraction property requires some explanations. First of all, as it is not difficult to see that this property is consistent with the standard 1D time. Indeed, we will have exactly the standard definitions if we take $m=1$ and fix $\Sigma=\Cal C$ as a non-negative semi-axis. On the other hand, it a priori looks more natural to require that $d(h,0)\ge D$. Actually, it is probably possible to develop the analogue of the attractors theory with such type of attraction, but it will be not so elegant, in particular, we will have problems with strict invariance of the attractor as well as with the structure of the set of complete trajectories (e.g., it may contain trajectories which are defined not in the whole $\R^m$), etc. Thus, the chosen way to extend the attraction property is the optimal one (at least from our point of view) and we will not consider the alternative ways in this paper.
\par
We also mention that it is possible to extend the theory (in the spirit of \cite{Z23}) to the case of general Hausdorff topological spaces $\Phi$ with general bornologies on them, non-continuous semigroup, etc, but in order to avoid technicalities, we will not consider such generalizations as well.
\end{remark}
We are now ready to state and prove our main result which guarantees the existence of such an attractor. 
\begin{theorem}\label{Th1.main} Let the multi-parametric semigroup $S(h):\Phi\to\Phi$, $h\in\Cal C$, acting on a  metric space $\Phi$ with a bornology $\Bbb B$ and a time arrow $\Sigma$, satisfy the following assumptions:
\par
1. Operators $S(h)$ are continuous for all $h\in\Cal C$;
\par
2. The semigroup $S(h)$ possesses a compact $(\Bbb B,\Sigma)$-attracting set $K$.
\par
Then there exists a $(\Bbb B,\Sigma)$-attractor $\Cal A\subset K$ of this semigroup;  
\end{theorem}
\begin{proof} When the  extensions of the main concepts of the attractors theory are fixed in a proper way, the proof of the theorem becomes very similar to the standard case of 1D time, see e.g. \cite{BV92}. Nevertheless, for the convenience of the reader, we present a detailed proof here. As usual, we first define the analogue of an $\omega$-limit set for a bounded set $B\subset\Bbb B$ as follows:
\begin{multline}
\omega(B):=\{u_0\in\Phi,\ \exists h_n\in\Sigma, \ u_n\in B,\ n\in\Bbb N, \\ \text{ such that }\ \lim_{n\to\infty}d(h_n,\partial\Sigma)=\infty,\ \ \lim_{n\to\infty}S(h_n)u_n=u_0\}.
\end{multline}
The equivalent  topological definition of  $\omega(B)$ reads
\begin{equation}\label{1.ot}
\omega(B)=\cap_{D\ge0}[\cup_{d(h,\partial\Sigma)\ge D}S(h)B]_{\Phi},
\end{equation}
where $[\cdot]_\Phi$ stands for the closure in $\Phi$. The proof of the equivalence is exactly the same as in 1D case and we left it to the reader.
\par
{\it Step 1.} $\omega(B)$ is not empty if $B$ is not empty and it is a subset of $K$. Indeed, let $h_n\in\Sigma$ be such that $d(h_n,\partial\Sigma)\to\infty$ and $u_n\in B$ be arbitrary (such sequences exist due to the assumption \eqref{2.sigma}). Due to the attraction property, up to passing to a subsequence, we may find a sequence $v_n\in K$ such that
$$
d(S(h_n)u_n,v_n)\le \frac1n.
$$
Utilising the compactness of $K$, we may assume without loss of generality that $v_n\to u_0\in K$. Finally, the triangle inequality gives the convergence $S(h_n)u_n\to u_0\in K$. Thus, $\omega(B)$ is not empty and, since the sequences $h_n$ and $u_n$ are chosen arbitrarily,  also see that $\omega(B)\subset K$. 
\par
{\it Step 2.} $\omega(B)$ attracts $B$. Indeed, assume that it is not so. Then, there exist a sequence $h_n\in\Sigma$ and a sequence $u_n\in B$ such that $d(h_n,\partial\Sigma)\to\infty$, but $d(S(h_n)u_n,\omega(B))\ge\eb_0>0$.  We have already proved that the sequence $S(h_n)u_n$ is precompact and therefore, up to passing to a subsequence, $S(h_n)u_n\to u_0\in\omega(B)$ (by the definition of $\omega(B)$). Thus, $d(u_0,\omega(B))=0$, but according to our assumption, $d(u_0,\omega(B))\ge\eb_0>0$ and this contradiction proves the statement. 
\par
{\it Step 3.} $S(h)\omega(B)=\omega(B)$ for all $h\in\Cal C$. Let us fix $h\in\Cal C$ and take arbitrary point $u_0\in\omega(B)$ together with the sequences $h_n\in\Sigma$ and $u_n\in B$ such that $d(h_n,\partial\Sigma)\to\infty$ and $S(h_n)u_n\to u_0$. Then, since $S(h)$ is continuous,
$$
S(h)u_0=S(h)\lim_{n\to\infty}S(h_n)u_n=\lim_{n\to\infty}S(h_n+h)u_n.
$$
Note that $d(h+h_n,\partial\Sigma)\to\infty$ as $n\to\infty$ and, due to the proved pre-compactness, we may assume without loss of generality that $S(h+h_n)u_n\to v_0\in\omega(B)$. This gives us the inclusion $S(h)\omega(B)\subset\omega(B)$.
\par
To prove the opposite inclusion, we note that, due to the assumptions that $d(h_n,\partial\Sigma)\to\infty$, the sequence $h_n-h$ will belong to  $\Sigma\subset\Cal C$ starting from a sufficiently big $n$ depending on $h$. Thus, without loss of generality, we may assume that $S(h_n-h)u_n\to w_0\in\omega(B)$ and, therefore, due to the continuity of $S(h)$,
$$
S(h)w_0=S(h)\lim_{n\to\infty}S(h_n-h)u_n=\lim_{n\to\infty}S(h_n)u_n=u_0.
$$
This finishes the proof of the strict invariance.
\par
{\it Step 4.}  Let us define
\begin{equation}\label{1.defatr}
\Cal A:=\bigg[\bigcup_{B\in\Bbb B}\omega(B)\bigg]_{\Phi}.
\end{equation}
We claim that $\Cal A$ is the desired attractor. Indeed, the attraction property is obvious since for any $B\in\Bbb B$, $\omega(B)\subset \Cal A$ and $\omega(B)$ attracts $B$. Compactness is also obvious since $\Cal A$ is a closed subset of the compact set $K$. Let us check the strict invariance. Indeed, let $u_0\in\Cal A$ and $h\in\Cal C$ be given. Then, there exist sequences $B_n\in\Bbb B$ and $u_n\in\omega(B_n)$  such that $u_0=\lim_{n\to\infty}u_n$. By the strict invariance of $\omega(B_n)$, we conclude that there are $v_n,w_n\in\omega(B_n)$ such that $S(h)v_n=u_n$ and $S(h)u_n=w_n$. Due to the compactness of $\omega(B_n)$, we may assume without loss of generality that $v_n\to v_0\in\Cal A$ and $w_n\to w_0\in\Cal A$. Finally, since $S(h)$ is continuous, we get $S(h)v_0=u_0$ and $S(h)u_0=w_0$. Thus, the strict invariance of $\Cal A$ is proved and the theorem is also proved.
\end{proof}
We now discuss the extra conditions which allow us to present the constructed attractor $\Cal A$ in terms of the set $\Cal K$ of complete bounded solutions.
\begin{corollary}\label{Cor1.main} Let the assumptions of Theorem \ref{Th1.main} hold and let in addition the attracting set $K$ be bounded ($K\in\Bbb B$). Then the following representation formula holds:
\begin{equation}\label{1.rep}
\Cal A=\Cal K\big|_{h=0},
\end{equation}
where $\Cal K$ is a set of all complete bounded trajectories of the semigroup $S(h)$.  
\end{corollary}
\begin{proof}This is the standard corollary of the strict invariance of the attractor. Indeed, the inclusion $\Cal K\big|_{h=0}\subset \Cal A$ is obvious, since $B_u$ is bounded and $u(-h_n)\in B_u$ for all $h_n$. To verify the opposite inclusion,  we take an arbitrary $u_0\in\Cal A$. Moreover, since $\Cal K$ is closed (in $C_{loc}(\R^n,\Phi)$), it is enough to consider the case when $u_0\in\omega(B)$ for some $B\in\Bbb B$ only, so $u_0=\lim_{n\to\infty}S(h_n)u_n$ for some $u_n\in B$ and $h_n\in\Sigma$ satisfying $d(h_n,\partial\Sigma)\to\infty$.   Then, the desired trajectory $u\in\Cal K$ is already defined in a cone $\Cal C$ via $u(s)=S(s)u_0$. Moreover, since the attractor is strictly invariant, we may find $u_1\in\Cal A$ such that $S(h_1)u_1=u_0$. This allows us to define the desired function $u(s):=S(s+h_n)u_1$ already for $s\in-h_n+\Cal C$ such that for $s\in\Cal C$ we have consistency with the previously defined function: $S(s+h_n)u_1=S(s)u_0$. Repeating this procedure and using that
$$
\cup_{n\in\Bbb N}(-h_n+\Cal C)=\R^m,
$$
which follows from $d(h_n,\partial\Sigma)\to\infty$ and the inclusion $\partial\Cal C\subset\partial\Sigma$,
we will construct the complete trajectory $u$ such that $u(0)=u_0$ and $u(h)\in\Cal A$ for all $h\in\R^m$. Since $\Cal A\subset K\in \Bbb B$ is bounded, we conclude that the trajectory $u\in\Cal K$ (we may take $B_u=K$).  This proves the representation formula and finishes the proof of the corollary.
\end{proof}
\begin{remark}\label{Rem1.cont-abs} Note that the assumption on the continuity of the operators $S(h)$ on the whole space may be too restrictive in applications. As follows from the proof of the theorem, it is enough, in a complete agreement with the one-parametrical case,  to require this continuity on the {\it absorbing} set of this semigroup only. A bit surprising that, it is not enough (again exactly as in the classical case) to require the continuity on the {\it attracting} set, see \cite{Z23} for the counterexample.
\par
We also note that the inclusion $\Cal K\big|_{h=0}\subset \Cal A$ holds without the assumption that $K\in\Bbb B$. In contrast to this, the opposite inclusion may fail (and often fails) without this extra assumption (the simplest example is the so-called point attractor where the bornology $\Bbb B$ consists of all one-point sets, see  \cite{Z23}.
\par
We  mention as well that, in the case where $K$ is bounded, the attractor possesses a simpler formula, namely, $\Cal A=\omega(K)$. This presentation  may fail in a general case.
\end{remark}
Let us now consider the case when $\Sigma=\Cal C$, fix any interior direction $\vec l\subset\Cal C$, $|\vec l|=1$ and introduce the directional semigroup $S_{\vec l}(t):=S(t\vec l)$. This is already a standard one-parametrical semigroup and we may speak about its $\Bbb B$-attractor $\Cal A_{\vec l}$. Under the assumptions of
Corollary \ref{Cor1.main}, this semigroup is continuous for every fixed $t$ and possesses the same compact attracting set $K$, therefore, there exists a $\Bbb B$-attractor $\Cal A_l$ of this semigroup. The next corollary tells that, under some natural extra assumption, this attractor coincides with the attractor of the whole semigroup $S(h)$ and, in particular, is independent of the choice of the direction $\vec l$.
\begin{corollary}\label{Cor1.dir} Let the assumptions of Corollary \ref{Cor1.main} be satisfied, $\Sigma=\Cal C$,  and let, in addition, the semigroup $S(h)$ be bounded on the compact attracting set $K\in\Bbb B$, i.e. the union $\cup_{h\in\Cal C}S(h)K$ is bounded. Then, for every interior direction $\vec l\in\Cal C$, the corresponding directional attractor $\Cal A_{\vec l}$ exists and is generated by all complete bounded trajectories of the initial semigroup $S(h)$, $h\in\Cal C$:
\begin{equation}\label{1.dir-rep}
\Cal A_{\vec l}=\Cal A=\Cal K\big|_{h=0}.
\end{equation}
\end{corollary}    
\begin{proof} Indeed, we only need to check the representation formula \eqref{1.dir-rep}. This can be done exactly as in the proof of Corollary \eqref{Cor1.main}. The only difference is when we take $u_0\in\Cal A_{\vec l}\subset K$, we know only that $S_{\vec l}(t)u_0$ is bounded, but we need the boundedness of $S(h)u_0$ for all $h\in\Cal C$. Since $\Cal A_{\vec l}$ is a priori not invariant with respect to the whole semigroup $S(h)$ (although a posteriori it is), we do not know how to prove this boundedness and have to put an extra assumption on the boundedness of $S(h)$. Then this boundedness will follow from the boundedness of $\Cal A_{\vec l}$ and the rest of the proof is identical. Thus, the corollary is proved.
\end{proof}
\begin{remark} As we see from the corollary, we can restore the whole attractor $\Cal A$ using a single interior direction $\vec l$ and classical one-parametric semigroups only. However, considering the attractor for the multi-parametric semigroup $S(h)$ gives an essential extra information, namely, the uniformity of attraction with respect to different directions $\vec l$, which is crucial for our applications to problems in the exterior domains, see next section, and this is one of our main sources of motivation to consider such attractors.
\end{remark}

\section{Applications}\label{s2}
In this section, we consider several applications of the developed theory to various PDEs in unbounded domains. We start with the simplest model examples of a parabolic equation in the whole space $\Omega=\R^d$.
\begin{example}\label{Ex1.par} Let us consider the following Cauchy problem:
\begin{equation}\label{2.par}
\Dt u=\Dx u+u-u^3,\ u\big|_{t=0}=u_0
\end{equation}
in the whole space $\Omega=\R^d$. It is well-known, see e.g. \cite{LU68, Z-VIN}, that this equation is well posed in the phase space $\Phi_b:=C_b(\R^d)$ and the following dissipative estimate holds:
\begin{equation}\label{2.dis}
\|u(t)\|_{C_b(\R^d)}\le  Q(\|u_0\|_{C_b(\R^d)})\chi(1-t)+C_*,
\end{equation}
where $\chi(z)$ is a standard Heaviside function and the positive constant $C_*$ and a monotone increasing function $Q$ are independent of $u_0$ and $t$. Note that estimate \eqref{2.dis} is a bit stronger than the standard dissipative estimate (the standard term $e^{-\alpha t}$ is replaced by $\chi(1-t)$). This is related with the super-linear growth rate of the nonlinearity and shows that any set of the initial conditions will be inside of the absorbing ball 
\begin{equation}\label{2.aball}
\Cal B:=\{u_0\in\Phi_b,\ \|u_0\|_{\Phi_b}\le C_*\}   
\end{equation}
for all $t\ge1$. The simplest way to see this is to consider an ODE
$$
y'=y-y^3, \ y(t)\in\Bbb R,
$$
for which the analogous property is obvious and can be obtained, e.g. from the explicit formula for the solution. Then, applying the comparison principle, we see that the same is true for the PDE \eqref{2.par} as well. A bit more accurate analysis shows that this property is related with the superlinear growth rate of the nonlinearity only and is not related with maximum/comparison principle. For instance, it can be obtained for a system of reaction-diffusion equations
$$
\Dt u=a\Dx u-f(u), \ \ u=(u_1,\cdots,u_n),
$$
where the diffusion matrix $a$ satisfies the condition $a+a^*>0$ and the nonlinearity $f(u)$ enjoying the superlinearity assumption:
\begin{equation}\label{2.sup-lin}
f(u).u\ge-C+|u|^{2+\eb}
\end{equation}
with $\eb>0$ and some growth restrictions, see \cite{Z-VIN} for the details.
\par
Important for our analysis is the fact that, due to the parabolic smoothing property, estimate \eqref{2.dis} implies the estimate
\begin{equation}\label{2.sdis}
\|u(t)\|_{C^1(\R^d)}\le C_{**},\ \ t\ge2,
\end{equation}
for some positive constant $C_{**}$ which is independent of $t$ and $u_0\in \Phi_b$. 
\par
To proceed further, we recall that the uniform topology of $\Phi_b$ is usually too strong to consider the attractors in unbounded domains and should be replaced by the local topology of $\Phi_{loc}:=C_{loc}(\R^d)$, see \cite{MZ08} and references therein. By this reason, we endow the space $\Phi_b$ with the topology induced by the embedding $\Phi_b\subset\Phi_{loc}$ and denote the obtained space by $\Phi$. Clearly, the space $\Phi$ is a metric space. Of course, it is not complete, but this is not important since, after the restriction of this topology to absorbing balls, it becomes complete. Actually, we fix the bornology $\Bbb B$, which consists of bounded sets in the metric space $\Phi_b$ and use the topology of $\Phi_{loc}$ for the attraction property. Moreover, we take $\Sigma=\Cal C$ here.
\par 
We recall that, since the Cauchy problem \eqref{2.par} is globally well-posed,  the solution semigroup $\Cal S(t):\Phi\to\Phi$, $\Cal S(t)u_0:=u(t)$, $t\ge0$, is well-defined. In addition, since our equation does not depend explicitly on $x$, the group of spatial shifts $T(s)u(x):=u(x+s)$,\ $s,x\in\R^d$, acts on the phase space $\Phi$ and commutes with the evolution semigroup $\Cal S(t)$:
$$
T(s)\circ \Cal S(t)=\Cal S(t)\circ T(s).
$$
 Therefore, the extended $(d+1)$-parametrical semigroup $S(h)$, $h=(t,s)\in\Cal C:=\R_+\times\R^d$, defined via
  \begin{equation}\label{2.ext}
  S(h):=\Cal S(t)\circ T(s), \ \ h\in\Cal C,
  \end{equation}
  acts on the phase space $\Phi$. Since we define the bornology on $\Phi$ using the usual bounded sets in the Banach space $\Phi_b$,  namely, $B\in\Bbb B$ if and only if $B$ is bounded in $\Phi_b$, it follows from the dissipative estimate \eqref{2.dis} that the ball $\Cal B$, defined by \eqref{2.aball}, is indeed a bounded absorbing set for this semigroup. However, we still unable to apply Theorem \ref{Th1.main} since this absorbing set is not compact in the topology of $\Phi_{loc}$. To get the compactness, we need to use the smoothing estimate \eqref{2.sdis}. Namely, let
  \begin{equation}\label{2.comp}
  K:=\bigg[\bigg\{ u_0\in C^1_b(\R^d),\ \|u_0\|_{C^1(\R_b)}\le C_{**}\bigg\}\bigg]_{\Phi_{loc}}.
  \end{equation}
Then, due to \eqref{2.sdis}, the set $K$ is also an absorbing set. On the other hand, due to the Arzela theorem, it is a compact set in $\Phi$. Note that, in contrast to the reflexive case, the closure in \eqref{2.comp} is necessary to get a compact set.   
\par
Thus, the existence of a compact absorbing set is verified. Let us discuss the continuity. Indeed, the space shifts $T(s)$ are obviously continuous in $\Phi$, but it seems that the evolution operators $\Cal S(t)$ are not continuous on the whole space $\Phi$. However, if we restrict them to the absorbing ball $\Cal B$, then they will be continuous, see e.g. \cite{Z04}. Moreover, the assumptions of Corollary \ref{Cor1.main} are also obviously satisfied.

Thus, due to Corollary \ref{Cor1.main} and Remark \ref{Rem1.cont-abs}, there exists a $\Bbb B$-attractor $\Cal A$ of the extended semigroup $S(h)$, $h\in\Cal C$ which is generated by the set $\Cal K$ of all complete bounded trajectories of this semigroup. Finally, as it is not difficult to see that $\Cal K$ consists of all bounded solutions of equation \eqref{2.par} defined for all $(t,x)\in\R^{d+1}$, Indeed, the one-to-one correspondence between them is constructed as follows: if $u(h)\in\Cal K$ then $u(t,0)$ is a complete bounded solution of \eqref{2.par}. Vice versa, if $u(t,x)$ is a complete solution, then $T(s)u(t,x)$ is a complete trajectory. Thus, in a complete analogy with one-dimensional time, the attractor $\Cal A$ of the extended semigroup $S(h)$ associated with equation \eqref{2.par} is determined by the set of all complete bounded (in $\Phi_b$) solutions of equation \eqref{2.par}.  
\par
Furthermore, for every direction $\vec l=(\kappa,\vec l_0)\in\Cal C$ with $\kappa>0$, we may define a directional semigroup $S_{\vec l}(\tau):=S(\tau\vec l)$, $\tau>0$. Then Corollary \ref{Cor1.dir} gives us that there exists an attractor of this semigroup and that $\Cal A_{\vec l}=\Cal A$. We see that the directional attractors $\Cal A_{\vec l}$ coincide with the constructed before attractor $\Cal A$ for all admissible $\vec l$. Nevertheless, the dynamics properties of $S_{\vec l}(\tau)$ strongly depend on $\vec l$. For instance, for purely temporal direction $\vec l=(1,0)$, $S_{\vec l}(\tau)=\Cal S(\tau)$ and we know that \eqref{2.par} is an extended gradient system, see \cite{Z04}. In particular, at least for small dimensions $d=1,2$, we do not have any periodic orbits and the dynamics is relatively simple. On the other hand, it is also known, see \cite{Z04}, that if $\vec l=(\kappa,\vec l_0)$ and $\kappa>0$ is large enough, there are plenty periodic orbits for $S_{\vec l}(\tau)$, moreover, it has infinite topological entropy and the dynamics is extremely chaotic, see \cite{Z04} for the details.
\end{example}
We now turn to elliptic equations, see \cite{Ba94,CMS93,MieZ02,VZ96, Z-VIN,Z98} for applications of the attractor theory for such equations.
\begin{example}\label{Ex2.ell-cone} Let $\Omega\subset\R^d$ be an unbounded domain with a sufficiently smooth boundary and let us consider the following Dirichlet boundary value problem in $\Omega$:
\begin{equation}\label{2.ell}
\Dx u-f(u)=0,\ \ u\big|_{\partial\Omega}=u_0.
\end{equation}
We assume that the nonlinearity $f\in C(\R,\R)$ satisfies assumption \eqref{2.sup-lin}. Moreover, we start with the standard case where the semi-invariant direction exists and assume that there exists a cone $\Cal C\in\R^d$ with a non-empty interior such that
\begin{equation}\label{2.trans}
\Cal T(s)\Omega\subset\Omega,\ \ s\in\Cal C.
\end{equation}  
where $\Cal T(s)(x):=x+s$ is the group of spatial shifts acting in $\R^d$. Finally, we assume that $u_0\in C_b(\partial\Omega)$.
\par
It is well-known that, under the above assumptions, problem \eqref{2.ell} possesses at least one solution $u\in \Phi_b:=C_b(\Omega)$ for every $u_0\in C_b(\partial\Omega)$ and any such  solution satisfies the following analogue of the dissipative estimate~\eqref{2.dis}:
\begin{multline}\label{2.ell-dis}
\|u\|_{C_b(B_1(x_0)\cap\Omega)}\le Q(\|u_0\|_{C_b(\partial\Omega)})\chi(2-d(x_0,\partial\Omega))+C_*,\  x_0\in\Omega,
\end{multline}
where $B_r(x_0)$ stands for the ball of radius $r$ in $\R^d$ centered at $x_0\in\R^d$, $\chi(x)$ is a Heaviside function and a positive constant $C_*$ and a monotone increasing function $Q$ are independent of $x_0$ and $u_0$. The proof of these facts can be found in \cite{Z-VIN}. We also mention that, combining estimate \eqref{2.ell-dis} with the standard interior regularity estimates, we get the following analogue of  \eqref{2.sdis}:
\begin{equation}\label{2.ell-int}
\|u\|_{C^1_b(B_1(x_0))}\le C_{**},\ \ x_0\in\Omega,\ \ d(x_0,\partial\Omega)\ge 3,
\end{equation}
for some new monotone function $Q$ and positive constant $C_{**}$. 
\par
We are now ready to construct a dynamical system (with multi-di\-men\-sio\-nal time) associated with problem \eqref{2.ell}. Note that the solution of this problem may be not unique, so it is natural, following  \cite{ChVi95,VZ96} to use the so-called trajectory dynamical systems and related trajectory attractors. Namely, let us denote by $\Cal K_+\subset \Phi_b$ the set of all solutions $u\in\Phi_b$ of problem \eqref{2.ell}which correspond to all initial data $u_0\in C_b(\partial\Omega)$.  Then, due to conditions \eqref{2.trans}, the translation semigroup
\begin{equation}\label{2.elsem}
(S(h)u)(x):=u(\Cal T(h)x)=u(x+h),\ \ h\in\Cal C,\ u\in \Cal K_+`
\end{equation}
acts on the space $\Cal K_+$, i.e. $S(h)\Cal K_+\subset \Cal K_+$, $h\in\Cal K_+$. We refer to this semigroup as a trajectory dynamical system $(S(h), \Cal K_+)$ associated with the elliptic equation \eqref{2.ell}. Analogously to the previous example, we endow the trajectory phase space $\Cal K_+$ with the topology induced by the embedding $\Cal K_+\subset C_{loc}(\Omega)$ and with the bornology induced by the embedding $\Cal K_+\subset\Phi_b$ and fix the time arrow $\Sigma=\Cal C$. Then, due to estimate \eqref{2.ell-dis}, the set
\begin{equation}
\Cal B:=\bigg\{u\in \Cal K_+,\ \sup_{x_0\in\Omega}\|u\|_{C(B_1(x_0)\cap \Omega)}\le C_*,\ x_0\in\Omega\bigg\}
\end{equation}
is a bounded absorbing set in $\Cal K_+$ for the semigroup $S(h)$. As in the previous example, this set is not compact in $\Cal K_+$, so we need to use the interior estimate \eqref{2.ell-int} and construct a compact absorbing ball via
\begin{equation}
K:=\bigg[\bigg\{u\in \Cal K_+,\ \|u\|_{C^1(B_1(x_0)\cap\Omega)}\le C_{**},\ x_0\in\Omega \bigg\}\bigg]_{\Phi_{loc}}.
\end{equation}
Indeed, the compactness of this set in $\Phi_{loc}$ follows from the Arzella theorem (as in the previous example), but, in contrast to the previous example, we need to check, in addition, that $K\subset\Cal K_+$. In other words, we need to check that if a sequence of solutions $u_n\in\Cal B$ of equation \eqref{2.ell} converges to $u\in\Phi_{loc}$ in the topology of the space $\Phi_{loc}$, then the limit function $u$ is also a solution of this equation. But the last fact follows immediately by passing to the limit in the definition of a weak solution. Thus, the set $K$ is indeed a compact absorbing set for the trajectory dynamical system $(S(h),\Cal K_+)$. The continuity of this semigroup is obvious since it is nothing more than the semigroup of spatial translations, so all of the assumptions of Theorem \ref{Th1.main} and Corollary \ref{Cor1.main} are satisfied and, therefore, the considered trajectory semigroup possesses a global attractor $\Cal A_{tr}\subset K$ which is now referred as a trajectory attractor of equation \eqref{2.ell} and this attractor is determined by the set $\Cal K$ of all bounded (in $\Phi_b(\R^d)$) solutions of equation \eqref{2.ell}:
\begin{equation}\label{2.ell-rep}
\Cal A_{tr}=\Cal K\big|_{x\in\Omega}.
\end{equation}
This result can be reformulated in the following form, which does not use explicitly the theory of the attractors. 
\begin{corollary}\label{Cor2.main} Let the domain $\Omega$ and the nonlinearity $f$ satisfy assumptions \eqref{2.trans} and \eqref{2.sup-lin} respectively and let $\Cal K$ be the set of all bounded (in $\Phi_b$) solutions of the elliptic problem \eqref{2.ell} defined for all $x\in\R^d$. Then, there exists a monotone increasing function $\alpha:\R_+\to\R_+$ such that $\lim_{z\to0}\alpha(z)=0$ and, for every solution $u\in \Phi_b$ defined in $\Omega$, the following estimate holds:
\begin{equation}\label{2.no-attr}
d_{C(B^1_x)}\(u\big|_{B^1_x},\Cal K\big|_{B^1_x}\)\le \alpha\(d_{\R^d}(x,\partial\Omega)\),
\end{equation}
for all $x\in\Omega$ such that $d_{\R^d}(x,\partial\Omega)>2$. Here 
 $d_V(p,U)$ stands for the distance between a point $p$ and a set $U$ in a metric space $V$ and the function $\alpha$ is independent of $u$ and $x\in\Omega$.
\end{corollary}
Indeed, this statement is a straightforward corollary of our choice of the attraction property to the attractor $\Cal A$, representation of the attractor through complete bounded solutions and estimate \eqref{2.ell-dis} which guarantees that all trajectories of our dynamical system $S(h)$ reach the absorbing set $\Cal B$ in a uniform finite time.  

\begin{remark}\label{Rem2.dir} As in the parabolic case, we may fix any internal direction $\vec l\in\Cal C$ and consider the directional semigroup $S_{\vec l}(\tau):=S(\tau\vec l):\Cal K_+\to\Cal K_+$. According to the general theory, its $\Bbb B$-attractor will coincide with the $\Bbb B$-attractor $\Cal A$ of the semigroup $S(h)$ constructed before. Such attractors have been considered previously in \cite{VZ96,Z-VIN}. However, the previous results for directional semigroups {\it do not} imply Corollary \ref{Cor2.main} since, although the attractor by itself is independent of the choice of the direction $\vec l$, the rate of the attraction to it a priori may depend on $\vec l$. The result concerning the attractors for semigroups with multi-dimensional time presented in this paper show that a posteriori the rate of attraction is uniform with respect to $\vec l$ and give us the nice attraction property \eqref{2.no-attr}.
\end{remark}
\end{example}
\begin{example}\label{Ex2.ext} Our next task is to extend the result of Corollary \ref{Cor2.main} to more general class of unbounded domains $\Omega$. Namely, we want to relax property \eqref{2.trans}. Indeed, on the one hand, it is not involved into the statement of Corollary \ref{Cor2.main} and it is difficult to find any reasons why this invariance property should be essential for the result, so we expect that this corollary remains valid for more general classes of domains, e.g. for exterior domains. On the other hand, if assumption \eqref{2.trans} is violated, there is no more dynamics on $\Cal K_+$ since the operators $(S(h)u)(x):=u(x+h)$ do not map any more the phase space $\Cal K_+$ into itself. By this reason, it seems difficult/impossible to get Corollary \ref{Cor2.main} for general unbounded domains using the dynamical approach and the attractors theory. Nevertheless, there is a non-trivial trick which makes this possible and which we will describe below. Indeed, let $\Omega\subset\R^d$ be an arbitrary domain with a sufficiently smooth boundary and let as before $\Cal K_+\subset\Phi_b:=C_b(\Omega)$
be the set of all bounded solutions of problem \eqref{2.ell} (which correspond to all possible values of the boundary data $u_0\in C_b(\partial\Omega)$). Then, every $u\in\Phi_b$ can be extended to some function $\tilde u\in \Phi_b(\R^d):=C_b(\R^d)$ (crucial that we do not assume that $\tilde u$ satisfies \eqref{2.ell} outside of $\Omega$). We denote by $\tilde{\Cal K}_+\subset\Phi_b(\R^d)$ the set of all functions such that $\tilde u\big|_{\Omega}$ solves \eqref{2.ell}.
\par
As the next step, we fix the trajectory phase space $\Phi:=\Phi_b(\R^d)$, endow this space by the topology of $\Phi_{loc}:=C_{loc}(\R^d)$  and also fix the translation group $(S(h)u)(x)=u(x+h)$ acting on it as the trajectory dynamical system. Obviously, $S(h)\Phi=\Phi$, $h\in\Cal C:=\R^d$ and the maps $S(h):\Phi\to\Phi$ are continuous for every fixed $h\in\R^d$.
\par
Up to the moment the construction looks useless since it is not related at all with the initial equation \eqref{2.ell}. However, equation\eqref{2.ell} comes into play through the proper choice of the bornology $\Bbb B$ on $\Phi$. Namely, we say that a subset $B\subset\Phi$ belongs to $\Bbb B$ if
\par
1. The set $B$ is bounded in $C_b(\R^d)$;
\par
2. $B\subset\tilde{\Cal K}_+$. In other words, the restriction of $B$ on $\Omega$ consists of solutions of the elliptic boundary value problem \eqref{2.ell}.
\par
We also fix the time arrow as $\Sigma:=\Omega$. Then, the following result holds.

\begin{proposition}\label{Prop2.main} Let the domain $\Omega\subset \R^d$ be such that condition \eqref{2.sigma} be satisfied with $\Sigma=\Omega$ and $\Cal C=\R^d$. Let also the phase space $\Phi$, the group $S(h):\Phi\to\Phi$, $h\in\R^d$, and the bornology $\Bbb B$ be defined as above. Then, this group possesses a $(\Bbb B,\Sigma)$-attractor, which is generated by the set $\Cal K$ of complete bounded (in $C_b(\R^d)$) solutions of \eqref{2.ell} defined on the whole space $x\in\R^d$, i.e. $\Cal A=\Cal K$.
\end{proposition}
\begin{proof} In order to apply our main Theorem \ref{Th1.main} to the group $S(h)$, $h\in\R^d$, we only need to construct a compact attracting set $K$ for it. We claim that 
\begin{equation}\label{2.comp1}
K:=\bigg[\bigg\{u\in \Phi,\ \|u\|_{C^1(B_1(x_0))}\le C_{**},\ x_0\in\R^d \bigg\}\bigg]_{\Phi_{loc}}.
\end{equation}
is the desired attracting set. As usual, its compactness follows from the Arzela theorem and we only need to verify the attraction property. Let $B\in\Bbb B$ be arbitrary and let $u\in\Bbb B$. Assume that $D>0$ is big enough and $h\in\R^n$ be such that $B_{D+3}(0)\subset\Omega$. Then, by the definition of $\Bbb B$, $T(h)u\big|_{B_{D+3}(0)}$ solves equation \eqref{2.ell} and, by the interior estimate \eqref{2.ell-int}
$$
\|u\|_{C^1(B_D(0))}\le C_{**}.
$$  
Let us take any $\tilde u\in\Phi$ such that $\tilde u\big|_{B_D(0)}=u\big|_{B_D(0)}$ and $\|\tilde u\|_{C_b(\R^d)}\le C_{**}$. Then,
$$
d_{\Phi_{loc}}(u,\tilde u)\le \alpha(D),
$$
where $\alpha:\R_+\to\R_+$ is a monotone increasing function satisfying the condition $\lim_{z\to0}\alpha(z)=0$, which is independent of $u$ and $B$, and $d_{\Phi_{loc}}(u,v)$ is some fixed metric in $C_{loc}(\R^d)$. Since $\tilde u\in K$, this proves the attraction property.
\par
Thus, due to Theorem \ref{Th1.main}, the considered semigroup $S(h):\Phi\to\Phi$ possesses a $(\Bbb B,\Sigma)$-attractor $\Cal A\subset K$. Moreover, from the general theory, we also have an inclusion $\Cal K\subset \Cal A$. However, in our case $K\notin\Bbb B$, so we cannot have the opposite inclusion from the general theory and need to verify it for our concrete case. Namely, we only need to verify that every $u\in \Cal A$ is a solution of \eqref{2.ell}. Let $u\in\Cal A$. Then, by Definition \ref{1.defatr} and the fact that the $C_{loc}(\R^d)$ limit of solutions is again a solution of \eqref{2.ell}, we may assume without loss of generality that $u\in\omega(B)$ for some $B\in\Bbb B$. This means that there exist a sequence $h_n\in\Omega$ such that $d(h_n,\Omega)\to\infty$ and $u_n\in B$ such that $u=\lim_{n\to\infty}S(h_n)u_n$. In particular, since $u_n\in B\in\Bbb B$, for every $D>0$, the functions $(S(h_n)u_n)\big|_{B_D(0)}$ solve equation \eqref{2.ell} if $h_n$ are large enough. Since the limit of solutions is a solution, we conclude that $u\big|_{B_D(0)}$ solves \eqref{2.ell} and since $D$ is arbitrary, we have $u\in\Cal K$. This finishes the proof of the proposition.
\end{proof}
Thus, we get the analogue of Corollary \ref{Cor2.main} for general unbounded domains.
\begin{corollary} \label{Cor2.good} Let $\Omega\subset\R^d$ be an unbounded domain with a sufficiently smooth boundary, \eqref{2.sigma} hold with $\Sigma=\Omega$ and let the nonlinearity $f$ satisfy \eqref{2.sup-lin}. Then, estimate \eqref{2.no-attr} holds for every solution $u\in C_b(\Omega)$.
\end{corollary}
Indeed,  the statement follows from Proposition \ref{Prop2.main}. 
\end{example}
\begin{example}\label{Ex2.disc} Note that in general it seems difficult to find explicitly the set $\Cal K$ of all complete bounded solutions $u\in\Phi_b(\R^d)$ of equation \eqref{2.ell}, but it is possible in some cases. In particular, due to the super-linearity of the function $f$, it is not difficult to prove that any solution $u\in C_{loc}(\R^d)$ is automatically bounded (i.e. $u\in C_b(\R^d)$, see \cite{Z97,Z-VIN} for details). In this examples we take the nonlinearity $f$ of the form:
\begin{equation}\label{2.good}
f(u)=u(u-1)^2\cdots(u-N)^2
\end{equation}
for some $N\in\Bbb N$. We claim that at least in low dimensional cases $d=1$ or $d=2$, the set $\Cal K$ has a very simple structure:
\begin{equation}\label{2.Kgood}
\Cal K=\{0,1,\cdots,N\}.
\end{equation}
In particular, the set $\Cal K$ is totally disconnected and does not possess any non-trivial solutions. Indeed, multiplying equation \eqref{2.ell} by $u$ (without integration) and denoting $U(x):=u^2(x)$, we arrive at
\begin{equation}\label{2.conv}
\Dx U=2|\Nx u|^2+f(u).u\ge0.
\end{equation}
Let $d=1$. Then the function $U(x)$ is convex and bounded, so it must be a constant. After that we see from \eqref{2.conv} that $|\Nx u|^2=0$, so $u$ is a constant. Finally, since $f(u).u=0$, we get the desired description of $\Cal K$.
\par
Let now $d=2$. Then, from equation \eqref{2.conv}, we conclude that $-U(x)$ is a bounded subharmonic function in $\R^2$ and, by the Liouville theorem, it must be a constant. The rest of the proof is the same as in 1D case.
\par
Obviously, for $d>2$, a bounded subharmonic function is not necessarily a constant, so we do not know whether or not \eqref{2.Kgood} holds if $d>2$, so we assume from now on that $d=2$ (the case $d=1$ is not very interesting since the corresponding equation can be solved explicitly). Combining \eqref{2.Kgood} with \eqref{Cor2.good}, we get the following result.

\begin{corollary} Let the smooth  domain $\Omega\in\R^2$ be such that the set
$$
\Omega_D:=\{x\in\Omega,\ \ d(x,\partial\Omega)>D\}
$$
is connected for a sufficiently large $D$. Then, for every $u\in C_b(\Omega)$ which solves the equation
$$
\Dx u=u(u-1)^2\cdots(u-N)^2,\ \ x\in\Omega,
$$
the exists a number $N_u\in\{0,1,\cdots,N\}$ such that 
$$
u(x)\to N_u,\ \ \text{as }\ d(x,\partial\Omega)\to\infty.
$$
\end{corollary}
\begin{proof} Indeed, the convergence to the set $\Cal K=\{0,1,\cdots,N\}$ follows from Corollary \ref{Cor2.good}. Let $D$ be such that the distance from $u(x)$ to $K$ is less than $\frac12$ if $d(x,\partial\Omega)>D$. Then, since $\Omega_D$ is connected and the set $\Cal K$ is totally disconnected, we must converge to a single point of $\Cal K$. 
\end{proof}
\begin{remark} If the set $\Omega_D$ is disconnected for all large $D$, we may have convergence to different elements of $\Cal K$ in different connected components. Note also that the equilibrium $u=0$ is hyperbolic and other equilibria $u=1,\cdots, N$ are not hyperbolic. At least in the case where $d=1$ or in cylindrical domains with one unbounded direction, there is a result that the presence of two or more hyperbolic equilibria guarantees that the set $\Cal K$ is not trivial, see \cite{FSV98}. We do not know whether or not this result remains true in multi-dimensional case.
\end{remark}
\end{example}

\end{document}